\documentclass[12pt]{article}

\usepackage{epsfig,amsmath,amsfonts,amssymb,latexsym,color,amsthm,hhline,multirow,subfigure,graphicx,bm}

\newtheorem{theorem}{Theorem}

\newtheorem{prop}{Proposition}
\newtheorem{lemma}{Lemma}

\newcommand{\E}{{\mathbb E}}
\newcommand{\RR}{{\mathbb R}}
\newcommand{\Rd}{\mathbb{R}^d}

\newcommand {\PP}{{\mathbb P}}
\newcommand{\sss}{\scriptscriptstyle}
\newcommand{\Zd}{\mathbb{Z}^d}

\newcommand{\0}{{\bf 0}}
\newcommand{\1}{{\bf 1}}
\newcommand{\cA}{\mathcal{A}}
\newcommand{\cR}{\mathcal{R}}
\newcommand{\vep}{\varepsilon}
\newcommand{\n}{{\bf n}}

\begin{document}

\title{Competition on $\mathbb{Z}^d$ driven by branching random walk}\parskip=5pt plus1pt minus1pt \parindent=0pt
\author{Maria Deijfen\thanks{Department of Mathematics, Stockholm University; {\tt mia@math.su.se}} \and Timo Vilkas \thanks{Mathematical Sciences, Gothenburg University; {\tt timo.vilkas@chalmers.se}} }
\date{March 2022}
\maketitle

\begin{abstract}
\noindent A competition process on $\mathbb{Z}^d$ is considered, where two species compete to color the sites. The entities are driven by branching random walks. Specifically red (blue) particles reproduce in discrete time and place offspring according to a given reproduction law, which may be different for the two types. When a red (blue) particle is placed at a site that has not been occupied by any particle before, the site is colored red (blue) and keeps this color forever. The types interact in that, when a particle is placed at a site of opposite color, the particle adopts the color of the site with probability $p\in[0,1]$. Can a given type color infinitely many sites? Can both types color infinitely many sites simultaneously? Partial answers are given to these questions and many open problems are formulated.

\vspace{0.3cm}

\noindent \emph{Keywords:} Branching random walk, asymptotic shape, competing growth, coexistence.

\vspace{0.2cm}

\noindent AMS 2010 Subject Classification: 60K35.
\end{abstract}

\section{Introduction}\label{sec:intro}

We consider a competition model on $\Zd$ driven by branching random walk (BRW). Let $\n=(n,0,\ldots,0)\in \Zd$. At time 0, a red particle is placed at the origin and a blue particle is placed at the neighboring site $\1$. The origin is also colored red and $\1$ is colored blue, while all other sites are initially uncolored. The red (blue) particles then reproduce according to BRW in discrete time with offspring law $\cR_r$ and $\cR_b$, respectively, where $\cR_r$ and $\cR_b$ are described in more detail below. When a particle is placed at an uncolored site, the site is assigned the same color as the particle and then keeps this color forever. If two particles of different color are placed at an uncolored site in the same time step, an arbitrary local tie-breaker is applied to decide the color of the site.

According to the above description, the two types evolve independently of each other. We generalize the model by introducing an interaction parameter $p$. Specifically, if a particle is placed at a site of opposite color, the particle switches color with probability $p\in[0,1]$. The case $p=0$ hence corresponds to a situation where the BRW:s evolve independently of each other and compete to reach the sites first. The case $p=1$ on the other hand corresponds to a situation where the particles switch color when placed at a site of opposite color, thereby giving an extra advantage to a type that has been successful in invading many sites. This phenomenon is qualitatively related to the latin phrase 'cuius regio, eius religio' referring to the principle that the ruler of the land dictated the religion of the inhabitants, in force in many European countries in the 16th and 17th century. We will sometimes use this phrase as a metaphor for the case $p=1$.

Before proceeding, we describe the reproduction of the types. Let $\mu_r$ and $\mu_b$ be probability distributions with support on non-negative integers, $K_r$ and $K_b$ arbitrary finite subsets of $\Zd$, and $\nu_r$ and $\nu_b$ probability measures on $K_r$ and $K_b$. We may assume that each site in $K_r$ ($K_b$) is assigned positive mass by $\nu_r$ ($\nu_b$), since otherwise the site(s) may be removed from the set. A particle that is born at time $t$ produces offspring at time $t+1$. The number of children of a red (blue) particle is determined by an independent draw from $\mu_r$ ($\mu_b$), and the children of a red (blue) particle at $x$ are placed at sites determined by repeated independent draws from $\nu_r$ ($\nu_b$) translated by $x$. Write $\cR_r=(\mu_r,\nu_r)$ and $\cR_b=(\mu_b,\nu_b)$ for the measures specifying the reproduction of the red and blue type, respectively. We will throughout make the following assumptions on $\cR_r$ and $\cR_b$:


\begin{itemize}
\item[$\bullet$] The measures $\mu_r$ and $\mu_b$ put mass only on strictly positive integers and have means strictly larger than 1. This means that each particle gives rise to at least one child, with a positive probability of two or more children, and implies that none of the processes dies out.
\item[$\bullet$] The sets $K_r$ and $K_b$ contain all neighbors of the origin. This is a simple way to ensure that any site can be reached by a particle after a finite number of steps. In order to make our questions non-trivial for $p=1$ in dimension $d=1$, we will in that case also assume that $K_r$ ($K_b$) contains at least one site in each direction that is not a neighbor of the origin.
\end{itemize}

Let $\PP_p(\cdot)$ denote the probability law of the process with parameter value $p$. Also, denote by $G_r$ and $G_b$ the events that infinitely many sites are colored red and blue, respectively, and define $C=G_r\cap G_b$. We are interested in the following questions:

\begin{itemize}
 \item[(i)] When are $\PP_p(G_r)$ and $\PP_p(G_b)$ strictly positive? When is one or both of them equal to 1? The answer may depend on the parameter $p$ and on the underlying reproduction laws $\cR_r$ and $\cR_b$.
 \item[(ii)] When do we have that $\PP_p(C)>0$? Obviously this requires that both $\PP_p(G_r)$ and $\PP_p(G_b)$ are strictly positive.
\end{itemize}

We will primarily give partial answers to (i), but along the way we also make some observations in the direction of (ii). First we consider the extremal cases $p=1$ and $p=0$. When $p=1$, a given type can strangle the other by surrounding it with an impenetrable layer of sites of its own color, implying that the surrounded type cannot color any more sites.

\begin{prop}[Cuius regio, eius religio]\label{prop:p1}
For all choices of $\cR_r$ and $\cR_b$, we have that $\PP_1(G_r\cap G_b^c)>0$ and $\PP_1(G_r^c\cap G_b)>0$.
\end{prop}

When $p=0$, the asymptotic growth of the corresponding single type BRW:s will be important for the outcome. It is well known that the set of sites where particles have been placed in a BRW grows linearly in time and converges to a deterministic asymptotic shape when scaled by time. Specifically, let $D(n)$ denote the set of sites where particles have been placed up to time $n$ in a BRW started with a single particle at the origin at time 0 and with reproduction $\cR=(\mu,\nu)$, satisfying the above assumptions. Let $\bar{D}(n)=\{x+(1/2,1/2]^d:x\in D (n)\}$ denote its embedding in $\Rd$. Then there exists a convex compact set $\cA$ with non-empty interior containing the origin such that almost surely, for any $\vep\in(0,1)$, we have that
\begin{equation}\label{eq:shape}
(1-\vep)\cA\subset \frac{\bar{D}(n)}{n}\subset(1+\vep)\cA
\end{equation}
for large $n$; see \cite[Theorem 1.10]{ComPop}. For $x\in\Rd$, we write $\tau^x$ for the distance from the origin to the boundary of $\cA$ in direction $x$, that is, $\tau^x$ is the asymptotic speed of the growth of a single type process in direction $x$. We mention that the inverse of $\tau^x$ is known as the time constant in direction $x$. Its existence follows from the subadditive ergodic theorem and is an integral part of the proof of the existence of an asymptotic shape.

Now consider a two-type process with $p=0$. The types then evolve according to their single type dynamics and it should come as no surprise that a type that is asymptotically faster in a given direction will win in that direction. As a consequence, if there are different directions where the power relationship between the types is reversed, then they will both be able to color infinitely many sites by dominating in different directions. If, on the other hand, one of the types is faster in all directions, it will defeat the other in that it will color all but a finite number of sites.

To state this formally, for $i\in\{r,b\}$, write $\cA_i$ for the asymptotic shape generated by a single type BRW with reproduction $\cR_i$, and $\tau_i^x$ for the associated speed in direction $x$. We say that blue is stronger than red, and red is weaker than blue, if $\tau_r^x<\tau_b^x$ for all $x$. Let $L^x$ be the half-line through $x$ starting at the origin and define $L_t^x= \{y\in L^x:|y|\geq t\}$. Finally, write $S_i$ for the set of sites that is ultimately colored by type $i\in\{r,b\}$ in the two-type process and let $\bar{S}_i=\{y+(1/2,1/2]^d:y\in S_i\}$.

\begin{prop}[Independent BRW:s]\label{prop:p0} Consider a process with $p=0$.
\begin{itemize}
\item[{\rm{(a)}}] If there exist $x$ such that $\tau_r^x<\tau_b^x$, then $\PP_0(G_b)=1$. Specifically, almost surely $\bar{S}_b\supset L_t^x$ for large $t$.
\item[{\rm{(b)}}] If there exist $x$ and $y$ with $\tau_r^x<\tau_b^x$ and $\tau_r^y>\tau_b^y$, then $\PP_0(C)=1$.
\item[{\rm{(c)}}] If blue is stronger than red, then $\PP_0(G_r^c\cap G_b)=1$.
\end{itemize}
\end{prop}

When $p=0$, the asymptotic growth of the types hence determines the outcome. In particular, a stronger type will defeat a weaker one by coloring all but a finite number of sites. For $p=1$, on the other hand, both types always have a possibility of outcompeting the other due to randomness in the beginning of the growth. One might ask which effect dominates for $p\in(0,1)$.


For $p<1$, but close to 1, we can show that both types still have the possibility of defeating the other, given that the types place their offspring according to the same spatial law (and hence differ only in the distribution of the number of offspring). We believe that the result is true also without this restriction; see the discussion around Lemma \ref{le:perturbed_lin}.

\begin{theorem}\label{prop:p_large} If $p$ is sufficiently close to 1 and $\nu_r=\nu_b$, then $\PP_p(G_r\cap G_b^c)>0$ and $\PP_p(G_r^c\cap G_b)>0$.
\end{theorem}

\begin{figure}
\centering
\includegraphics[scale=0.8]{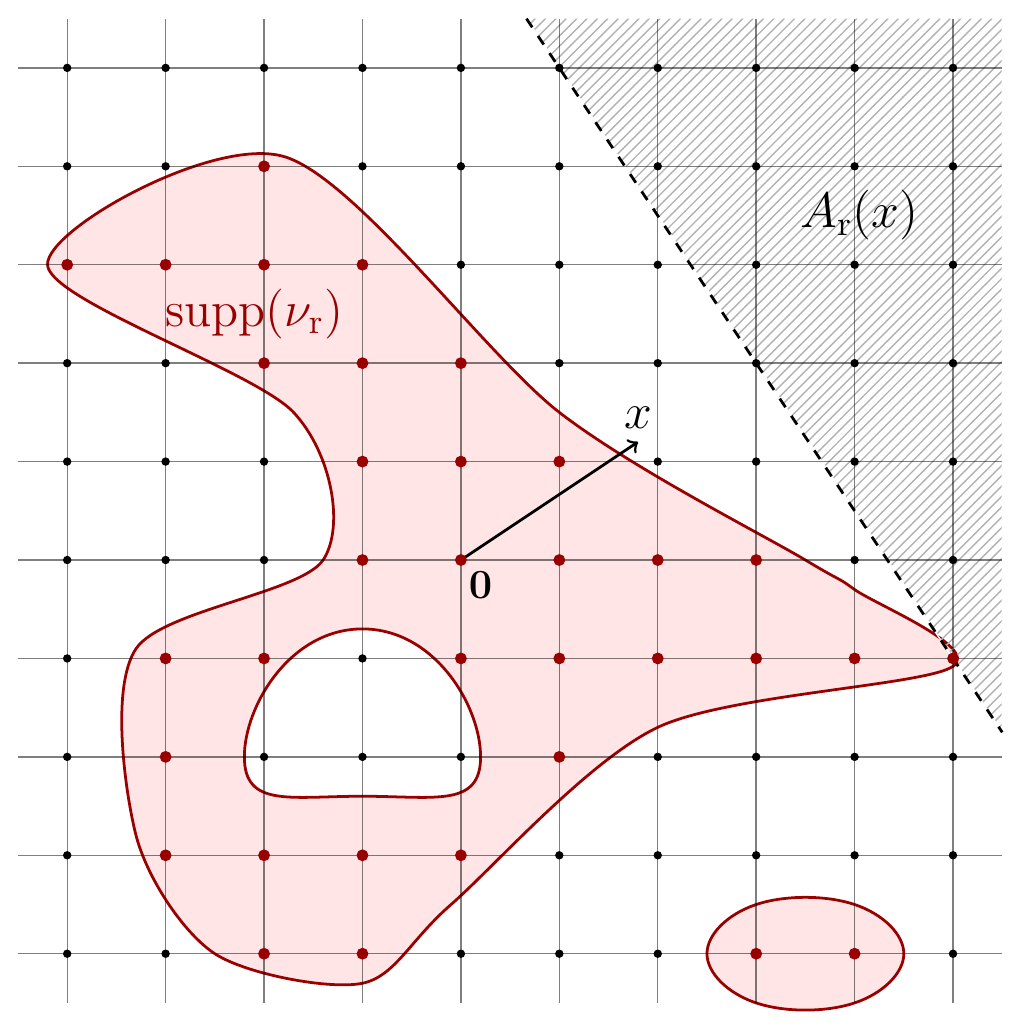}
\caption{Maximum range of the red offspring displacement $\nu_\mathrm{r}$ in a given direction $x\in\Zd$.}\label{range}
\end{figure}

For $p>0$, but close to 0, a type that has a sufficiently big advantage will still almost surely occupy infinitely many sites, but we cannot guarantee that the other type does not do so as well. Furthermore, the concept of having an advantage is stronger than for $p=0$. To specify what we need, let us first define the reach of an offspring displacement in a given direction. For $x\in \RR^d$, let $\rho_r(x):=\max\{\langle x,y\rangle;\;y\in\mathrm{supp}(\nu_r)\}$, where $\langle \cdot,\cdot\rangle$ denotes scalar product. Furthermore, define the open set $A_r(x):=\{y\in\Zd;\;\langle x,y\rangle>\rho_r(x)\}$, see Figure \ref{range} for an illustration. Note that $\rho_r(x)$ depends both on the direction and norm of $x$, while $A_r(x)$ depends only on the direction of $x$. Also note that, for all directions $x\in \RR^d$, we have that $\nu_r(A_r(x))=0$ by definition. Finally, for any set $A\subseteq \Zd$, write $n_r(A)$ for the expected number of offspring placed in $A$ by a red particle at the origin, that is, $n_r(A):=\E[\mu_r]\nu_r(A)$. Define $\rho_b(x)$, $A_b(x)$ and $n_b(A)$ analogously for blue. We now say that blue has a {\em supercritical advantage} over red in direction $x\in\RR^d$, if $n_b(A_r(x))>1$, that is the expected number of blue offspring placed further in direction $x$ than red offspring can reach is larger than 1.

It turns out that, if a type has a supercritical advantage in some direction, it will almost surely color infinitely many sites. The macroscopic assumption of a larger asymptotic speed in Proposition \ref{prop:p0}(a) is hence replaced by a microscopic assumption on the reproduction laws. The idea is that, if a type has a supercritical advantage over the other, a branching process can be defined that is supercritical for $p$ sufficiently small and where survival of the process implies that the favored type reaches untouched land in each time step after finitely many steps.

\begin{theorem} \label{prop:p_small} If there exist $x\in\RR^d$ such that $n_b(A_r(x))>1$, then $\PP_p(G_b)=1$ for $p\in\big[0,\frac{n_b(A_r(x))-1}{n_b(A_r(x))}\big)$. If, in addition, there exist $y\in\RR^d$ such that $n_r(A_b(x))>1$, then $\PP_p(C)=1$ for $p>0$ sufficiently small.
\end{theorem}

We also note that having a supercritical advantage is a stronger assumption than having a larger speed in some direction, in the sense that the former implies the latter.

\begin{prop}\label{le:supercritical} If there exist $x$ such that $n_b(A_r(x))>1$, then there exist $x'$ such that $\tau_b^{x'}>\tau_r^{x'}$.
\end{prop}

The rest of the paper is organized so that the rest of this section consists of some suggestions of further work and a short overview of related work. The proofs are then given in two separate sections, one for $p=1$ and $p\lessapprox 1$, and one for $p=0$ and $p\gtrapprox 0$.

\subsection{Open problems}\label{sec:open}

Several aspects of our questions (i) and (ii) are left open. Here we elaborate on some of them and describe possible extensions of the model.

\textbf{Behavior for $p$ close to 0.} Our result for $p\gtrapprox 0$ states that a type that has a supercritical advantage almost surely colors infinitely many sites. Does the other type capture only a finite number, or can it also grow large? Does the statement remain true if the advantage is only in terms of the asymptotic shape, as for $p=0$? A weaker type (in terms of the shape) cannot occupy infinitely many sites when $p=0$, but when $p=1$ it can. It would be interesting to understand if this possibility arises abruptly when $p$ becomes positive.

\textbf{Behavior for $p\in(0,1)$.} How does the process behave for $p$ that is not close to neither 0 nor 1? For $p=0$ and $p\gtrapprox 0$, a type that has some type of advantage will almost surely color infinitely many sites, while for $p=1$ and $p\lessapprox 1$ there is a positive probability that it colors only finitely many sites. When does this change and what explains the change?

\textbf{Coexistence in the extremal cases.} As for the question (ii) concerning coexistence, our only result so far is that coexistence has positive probability when $p$ is small and the types dominate in different directions; see Proposition \ref{prop:p0}(c) and Theorem \ref{prop:p_small}. Is coexistence possible when $\cA_r=\cA_b$ and $p=0$? We believe the answer is yes, at least in the symmetric case when the types have the same reproduction law. One might guess that coexistence then happens with probability 1, but randomness in the beginning of the growth could potentially give one of the types an impregnable lead. This however needs to be further investigated. When $p=1$, we believe that coexistence in a similar fashion is possible in the totally symmetric case and when the types dominate in different directions, while it is impossible when one of the types is stronger than the other.

If our suggestions are correct, the possibility of coexistence behaves similarly in the two extremal cases. We stress however that the geometric properties of the set of sites colored by the respective types is presumably very different. In the symmetric case for instance, coexistence when $p=1$ most likely occurs in that the types dominate in different regions, while for $p=0$ coexistence occurs in that both types color sites close to the boundary of their joint asymptotic shape, resulting in a mix of both colors.

\textbf{More general reproduction.} In our setup, the types place their offspring independently in finite sets. This could be generalized to allow for infinite spatial range, and for dependence in the placement of the children of a given particle. Also the number of children and their placement could be allowed to be dependent. Such more general reproduction laws would make our proofs longer and more technical, but the results should generally still be valid. However, some assumptions controlling the spatial growth of the process will be needed, for instance to ensure that there exists an asymptotic shape.

\subsection{Related work}

Models for competition on $\Zd$ have been studied for approximately two decades. One of the first examples is a two-type version of first passage percolation introduced in \cite{HP1}, where the competition is driven by i.i.d\ passage times on the edges, with potentially different distributions for the two types. The case with exponential passage times is known as the Richardson model and has received particular attention. It is clear that each type has a positive probability of occupying infinitely many sites by strangling the other, thereby preventing it from growing any further. The main question is whether coexistence is possible. For the Richardson model the answer is believed to be yes if and only if the infections have the same intensity; see \cite{pleasures} for an overview and further references. Versions of the two-type Richardson model have been considered for instance in \cite{Neu}, where sites recover after some time, in \cite{urns}, where a site with at least two neighbors of the same type is immediately occupied by that type, and in \cite{VA}, where sites reached by a one-type process may mutate into a different species. Our model with $p=1$ is qualitatively similar to competing first passage percolation in that one of the types may win by surrounding the other.

Another type of competition models is provided by two-type versions of growth models driven by moving particles. Here the type is not associated with the sites, but with particles moving on the sites. The growth typically starts from an i.i.d.\ configuration of inactive particles and when a particle is activated it starts moving according to simple random walk. If particles are inactive until they are hit by active particles the model is known as the frog model, while a version where all particles are active from the start is referred to as the diffusive epidemic model. Two-type version of these models are analyzed in \cite{comp_frogs} and \cite{KS_03,RW_competition}, respectively. Our model shares some features of these models, but differ in that particles are not present initially, but arise as a result of the growth, and in that type is assigned both to particles and sites.

BRW has been a very active topic in contemporary probability the last two decades; see \cite{Shi} for a survey covering mainly the one-dimensional case. BRW in higher dimensions is less well understood, but shape theorems can be found in \cite{Biggins,ComPop}. The model is well suited to describe spatial evolution of biological populations and versions of the model incorporating competition have been analyzed in this context; see e.g.\ \cite{AP, BP, BD, Eth}. The competition in these models however amounts to a single type of particles competing with each other in that there are constraints on the particle density or mass. An example of a two-type competition model is provided by \cite{BEM}, where the number of particles in bounded regions is limited. In our model, there are no limitations on the particle density, but competition arises in that the first type to reach a site is given a perpetual local advantage. It thus combines aspects of all of the above model types.

\section{'Cuius regio, eius religio' and $p$ close to 1}

We begin with the simple proof of Proposition \ref{prop:p1}.

\begin{proof}[Proof of Proposition \ref{prop:p1}]
To show that $\PP_1(G_r^c\cap G_b)>0$, consider a scenario where the red type (starting at the origin) places all its offspring at $\1$ in the first time step, while the blue type (starting at $\1$) does not place any offspring at the origin. Since $\1$ is colored blue from the start and $p=1$, this means that there will be no red particles immediately after the first time step, but the only presence of red is that the origin is colored red. Blue then proceeds to color all sites in $K_r$ without ever placing any offspring at the origin. This entails that no further sites can be colored red, since no uncolored sites are within reach for red particles placed at the origin (arising when blue particles place offspring there). By the assumptions on $\cR_r$ and $\cR_b$ -- in particular the fact that $K_r$ is finite and $\nu_r$ and $\nu_b$ put positive mass on all nearest neighbors of the origin -- this scenario has positive probability, which proves the claim. In $d=1$, we here also need the assumption that offspring is not placed only at nearest neighbors. That $\PP_1(G_r\cap G_b^c)>0$ is proved analogously with the roles of red and blue interchanged.
\end{proof}

To deal with the case when $p\lessapprox 1$, we will need two auxiliary results. The first one states that a one-type process where particles are removed with probability $p$ is unlikely to reach sites far from the origin if $p$ is large.

\begin{lemma}\label{GW} Consider a one-type process, starting with one (red) particle at the origin and where each new particle is removed during branching independently with probability $p$. The probability of any offspring reaching a site at distance $cn$ from the origin is bounded from above by $\big((1-p)\E[\mu_r]\big)^{\frac{cn}{\rho_r}}$, where $\rho_r=\max\{\rho_r(x);\;|x|=1\}$ and $c,n>0$.
\end{lemma}

\begin{proof} In order to reach a site at distance $cn$ from the origin, the process has to survive for at least $\lceil\frac{cn}{\rho_r}\rceil$ generations, since the displacement in each branching is bounded by $\rho_r=\max\{\rho_r(x);\;|x|=1\}$. If $X_k$ denotes the number of particles in generation $k$, its expectation is $\E[X_k]= \big((1-p)\E[\mu_r]\big)^k$. A simple first moment method argument then gives that $\PP(X_k>0)\leq \E[X_k]$, which proves the claim.
\end{proof}

In the second auxiliary result, we want to establish that (a version of) the two-type process grows linearly in time in the sense that the (continuum version of the) set of visited sites contains a linearly growing ball. Note that this is not an immediate consequence of the more general version of the shape theorem in \cite{ComPop}, which allows for random environment in the branching, since our environment is not i.i.d.\ but arises from the evolution of the process. Furthermore, the branching is based on particle type rather than the color of the site. Although the statement might appear obvious, we have found it difficult to establish rigorously, and have settled for the special case when the types have the same spatial reproduction law, where it follows from the one-type shape theorem. Here, $B(s)$ denote a ball in $\RR^d$ with radius $s$ centered at the origin.


\begin{lemma}\label{le:perturbed_lin}
Consider a two-type process which starts with one (red) particle at the origin and some sites pre-colored in an arbitrary way. If $\nu_r=\nu_b$, then the set of visited sites grows linearly. Specifically, there exist $c>0$ such that for any $\vep>0$ we have that $\PP(\bar{D}(n)\supset B(cn))\geq 1-\vep$ for $n\geq N$, where $c$ and $N=N(\vep)$ may be chosen independently of the pre-coloring.
\end{lemma}

\begin{proof}
Let $\mu$ denote the probability measure arising from taking the minimum of a draw from $\mu_r$ and an independent draw from $\mu_b$, and note that $\mu$ inherits the properties that it puts positive mass only on strictly positive integers and has mean strictly larger than 1 from $\mu_r$ and $\mu_b$. Consider a one-type process with reproduction law $\cR=(\mu,\nu)$, where $\nu:=\nu_r=\nu_b$ denotes the common spatial law of the types, and let $D_{\sss\textrm{min}}(n)$ denote the set of visited sites at time $n$. It is straightforward to confirm that our two-type process can be coupled to this one-type process in such a way that $D(n)\supset D_{\sss\textrm{min}}(n)$, regardless of the pre-coloring. The one-type process fulfills the assumptions of the shape theorem from \cite{ComPop} and the claim hence follows from the lower bound in \eqref{eq:shape}.
\end{proof}

\noindent\textbf{Remark.} We remark that there are other assumptions on the reproduction laws that also guarantee that the two-type process can be bounded from below by a one-type process that fulfills the assumptions of the shape theorem. A particle could for instance always place a copy of itself at its birth location and, with a strictly positive probability, produce at least one additional offspring placed independently in the bounded (possibly different) sets $K_r$ and $K_b$, respectively, containing all neighbors of the origin. Different sets $K_r$ and $K_b$ can also be allowed under the assumption that the reproduction laws are ordered in that, for each site $z\in\Zd$, the number of particles placed at $z$ by $\nu_r$ is stochastically smaller than for $\nu_b$.
\hfill$\Box$

With Lemma \ref{GW} and \ref{le:perturbed_lin} at hand, we proceed to prove Theorem \ref{prop:p_large}.

\begin{proof}[Proof of Theorem \ref{prop:p_large}] We show that $\PP(G_r^c\cap G_b)>0$. To this end, first note that, with small but positive probability, the entire offspring of the initial red particle lands on the initially blue site $\1$ and gets recolored there, all (blue) particles that land on the origin before a fixed time $t>0$ avoid being recolored and stay blue, the blue particles visit all sites within distance $d$ from the origin up to time $t-1$ and at time $t$ a single particle lands on the origin and gets re-colored red. We may also assume that this scenario is achieved under the restriction that no particle produces more than $m$ children, where $m$ is an arbitrary fixed integer exceeding the smallest possible value for both $\mu_r$ and $\mu_b$. Let us refer to this event as $I=I(t,d,m)$, where $t$ and $d$ will be specified below.

Now consider a two-type process started from the single red particle at the origin at time $t$, where the origin is red and all other sites visited by the process up until then are blue. Write $E_N$ for the event that all sites in $B(cn)$ have been visited by this process for all $n\geq N$, where $c$ and $N$ are chosen such that $\PP(E_N)> 1/2$ for all initial colorings -- this is possible by Lemma \ref{le:perturbed_lin}. Then set $d=cN$ and choose $t$ accordingly big such that the initial scenario $I$ sketched above happens with positive probability. Finally, consider each red particle at the origin at time $t+n$ in our original two-type process as a seed of a new process as in Lemma \ref{GW} (where each offspring is removed with probability $p$), write $R_n$ for the event that at least one of them produces offspring that reach a site at distance $cn$ from the origin and let $R=\cup_{n=1}^\infty R_n$.

We claim that, on $I=I(t,cN,m)$, the event $E_N\cap R^c$ guarantees that no site except for the origin is colored red in the two-type process, that is,
\begin{equation}\label{eq:lb}
\PP(G_r^c\cap G_b)\geq \PP(E_N\cap R^c) \mbox{ on }I.
\end{equation}
To see this, first note that $R^c$ guarantees that no red particle located at the origin at time $s\leq t+N$ produces red offspring at an uncolored site, since $B(cN)$ is blue already at time $t$ on $I$ and all red particles placed at sites in $B(cN)$ are hence colored blue with probability $p$. Assume inductively that no red particle placed at the origin at any time $s\leq t+n$ gives rise to offspring at an uncolored site, where $n\geq N$. The event $E_N$ then guarantees that all sites in $B(cn)\setminus \{\0\}$ are blue at time $t+n$ in our two-type process. Red particles placed at sites in $B(cn)$ are hence recolored with probability $p$, and $R_n^c$ then guarantees that no red particle located at the origin at time $t+n+1$ gives rise to red offspring at uncolored sites.

By \eqref{eq:lb}, we are done if we show that $\PP(R^c)> 1/2$ on $I$ for $p$ sufficiently close to 1, since then $\PP(E_N)+\PP(R^c)>1$, implying that $\PP(E_N\cap R^c)>0$. To estimate $\PP(R^c)$, write $M_n$ for the number of red particles at the origin at time $t+n$ in the two-type process and note that, by Lemma \ref{GW}, we have that
\begin{equation}\label{eq:R_est}
\PP(R)\leq \sum_{n=1}^\infty \PP(R_n)\leq \E [M_n] \big((1-p)\E[\mu_r]\big)^{\frac{cn}{\rho_r}}.
\end{equation}
Let us crudely bound $\E[M_n]$ on $A$: Note that, on $I$, the number of particles at any given site at time $t$ is at most $2m^t$, since each particle produces at most $m$ children. Consider a one-type process that initially has $2m^t$ particles at every site and when branching gives rise to $X+Y$ particles, where $X\sim\mu_r,\ Y\sim\mu_b$ are independent, and the $X$ particles are displaced according to $\nu_r$, the $Y$ particles displaced according to $\nu_b$, that is, an independent superposition of red and blue branching with initially $2m^t$ particles per site. Clearly such a process dominates the two-type process starting from the configuration prescribed by $I$ in terms of particles per site at any given time. Since it is started from a translation invariant configuration, the distribution of this one-type process is translation invariant. The random measure defining the mass sent out from a given site in a given time step as the number of offspring produced by particles at the site is hence translation invariant and the mass-transport principle then implies that the expected mass received by the site equals the expected mass sent out; see e.g.\ \cite[p.43]{BLPS}. From this we deduce that $\E[M_n]\leq [2m^t(\E[\mu_r]+\E[\mu_b])]^n$ on $I$.

Now choosing $p$ sufficiently close to 1 will make $2m^t(\E[\mu_r]+\E[\mu_b])\cdot\big[(1-p)\cdot \E[\mu_r]\big]^{\frac{c}{\rho_r}}$ smaller than 1, ensuring that the sum in \eqref{eq:R_est} is convergent. By choosing $p$ even larger we can make the sum smaller than 1/2, completing the proof.
\end{proof}

\section{Independent BRW:s and $p$ close to 0}

We first prove Proposition \ref{prop:p0}, which essentially follows from the one-type shape theorem.

\begin{proof}[Proof of Proposition \ref{prop:p0}]
To show (a), fix $\vep\in(0,1)$ such that $(1-\vep)\tau_b^x>(1+\vep)\tau_r^x$, and define $t_b(n)=(1-\vep)n\tau_b^x$ and $t_r(n)=(1+\vep)n\tau_r^x$. For $i\in\{r,b\}$, let $D_i(n)$ denote the set of sites where type $i$ particles have been placed up to time $n$ and set $\bar{D}_i(n)=\{y+(1/2,1/2]:y\in D_i(n)\}$. By \eqref{eq:shape}, we have that $\bar{D}_r(n)\cap L_{t_b(n)}^x=\emptyset$ for $n>N_r$ where $N_r<\infty$ almost surely, that is, the red type does not reach further than $t_r(n)$ in direction $x$ for large $n$. However \eqref{eq:shape} also implies that $L^x\setminus L^x_{t_r(n)}\subset \bar{D}_b(n)$ for $n>N_b$, that is, blue has covered everything up to $t_b(n)$ in direction $x$ for large $n$. Recall that $S_b$ denotes the set of sites that are ultimately colored red. Since a site is colored by type $i$ if type $i$ is the first one to place offspring there, it follows that $L^x_t\subset \bar{S}_b$ for $t>t_b(N)$, where $N=\max\{N_r,N_b\}$.

Part (b) is an immediate consequence of (a). To prove (c), denote $\inf\{\tau_b^x-\tau_r^x:|x|=1\}=\lambda'$ and note that, since $\cA_r$ and $\cA_b$ are compact, we have that $\lambda'>0$. Furthermore, denote $\lambda=\lambda'/\sup \{|x|:x\in\cA_r\cup \cA_b\}$. By the definition of $\lambda$, we have that $(1+\lambda/3)n\cA_r\subset (1-\lambda/3)n\cA_b$. Furthermore, it follows from \eqref{eq:shape} that $\bar{D}_b(n)\supset (1-\lambda/3)n\cA_b$ for $n>N_b$, while $\bar{D}_r(n)\subset(1+\lambda/3)n\cA_r$ for $n>N_r$. Since a site is colored by the type that reaches it first, we conclude that red will not color any sites after time $N=\max\{N_r,N_b\}$.
\end{proof}

We finally treat the case where $p\gtrapprox 0$.

\begin{proof}[Proof of Theorem \ref{prop:p_small}] To begin with, note that $\{\langle x,y\rangle,y\in\Zd\}$ is a discrete subset of $\RR$, with values being at least $\min\{|x_i|;\;1\leq i\leq d, x_i\neq 0\}$ apart. Consequently, all displacements $y\in K_b\cap A_r(x)$ of blue offspring actually fulfills $\langle x,y\rangle\geq\rho_\mathrm{r}(x)+\epsilon$ for some $\epsilon>0$.

Now, let us consider a thinned version of the blue BRW, consisting only of offspring having a displacement in $A_r(x)$ and not getting recolored, and compare it with a branching process, in which the offspring distribution $\mu_b$ is thinned out by (independently) keeping every newborn particle with probability $(1-p)\nu_b\big(A_r(x)\big)$. Obviously, with respect to the number of particles, the latter is stochastically dominated by the former and by our assumptions it is supercritical, that is, it has a positive probability to produce an infinite progeny, as the expected number of offspring per particle amounts to
\[
\E[\mu_b]\nu_b\big(A_r(x)\big)(1-p)=n_b\big(A_r(x)\big)(1-p)>1.
\]
Since $\langle x,y\rangle\geq \rho_r(x)+\epsilon$ for any displacement $y\in A_r(x)$ of these particles, provided it does not die out, any such thinned blue BRW will outgrow both the initial BRW of red particles as well as any (red) progeny of recolored blue particles (which have displacement at most $\rho_r(x)$ per generation) in direction $x$. Its advantage in direction $x$ will therefore enable it to visit infinitely many sites before any red particle does and this conclusion actually holds true irrespectively of the site that the considered thinned blue BRW originates from and the time of its emergence: Let us say that it originates at time $t_0$ from site $z_0$. Then there will be sites $z$ visited by the progeny of this thinned BRW at time $t$ with $\langle x,z\rangle\geq(\rho_r(x)+\epsilon)\cdot(t-t_0)-\langle x,z_0\rangle$ for all $t\geq t_0$. For a red particle to reach such a site $z$ earlier, there needs to be a chain of sites $y_0,y_1, \dots, y_n$ (not necessarily all marked red) such that $n\leq t$, $y_0=\mathbf{0}$, $y_n=z$ and $\langle y_i-y_{i-1},x\rangle\leq\rho_r(x)$ for all $1\leq i\leq n$, hence $\langle x,z\rangle\leq t\cdot \rho_r(x)$. For large enough $t$ this leads to a contradiction.

In order to conclude, we simply have to verify that there will a.s.\ emerge such a thinned blue BRW that does not die out. Since site $\mathbf 1$ is marked blue and the total number of particles (joint red and blue) goes to infinity, which together with our assumptions on the displacements make the joint process recurrent, the number of blue particles ever seen at site $\mathbf 1$ grows to infinity a.s. Each such particle starts an infinite thinned blue BRW with positive probability, and the claim hence follows from the conditional Borel-Cantelli-Lemma.
\end{proof}

We end by confirming that having a supercritical advantage in some direction implies an advantage in terms of asymptotic speed in some (possibly different) direction.

\begin{proof}[Proof of Proposition \ref{le:supercritical}] Assume that blue has a supercritical advantage in some direction $x$ and consider the thinned blue BRW described in the proof of Proposition \ref{prop:p_small}, where only offspring placed in $A_r(x)$ is considered. This BRW survives with positive probability and, if it does, we have that $\mathcal{A}_b\cap A_r(x)\neq \emptyset$, where $\mathcal{A}_b$ is the asymptotic shape of a blue one-type process. Since the convergence in the shape theorem is with probability 1, we conclude that $\mathcal{A}_b\cap A_r(x)\neq\emptyset$ almost surely. On the other hand, we have by the definition of $A_r(x)$ that $\mathcal{A}_r\cap A_r(x)=\emptyset$. It follows that the blue shape must exceed the red shape in some direction $x'$.
\end{proof}


\begin{thebibliography}{99}

\bibitem{AP} Addario-Berry, L.\ and Pennington, S.\ (2017): The front location in branching Brownian motion with decay of mass, \emph{Ann. Probab.} \textbf{45}, 3752-3794.

\bibitem{urns} Ahlberg, D., Griffiths, S., Janson, S. and Morris, R. (2018): Competition in growth and urns, \emph{Rand. Struct. Alg.} \textbf{54}, 211-227.

\bibitem{BLPS}Benjamini, I.,  Lyons, R.,  Peres, Y. and Schramm, O. (1999): {Group-invariant percolation on graphs,}{\em GAFA} {\bf 9}, 29-66.

\bibitem{Biggins} Biggins, J.D.\ (1978): The asymptotic shape of the branching random walk, {\em Adv.\ Appl Prob.} {\bf 10}, 62-84.

\bibitem{BEM} Blath, J., Etheridge, A.\ and Meredith, M.\ (2007): Coexistence in locally regulated competing populations and survival of branching annihilated random walk, \emph{Ann. Appl. Probab.} \textbf{17}, 1474-1507.

\bibitem{BP} Bolker, B.\ and Pascala, S.\ (1997): Using moment equations to understand stochastically driven spatial pattern formation in ecological systems, \emph{Theor. Popul. Bio.} \textbf{52}, 179-197.

\bibitem{BD} Brunet, E.\ and Derrida, B.\ (1997): Shift in the velocity of a front due to a cutoff, \emph{Phys. Rev. E} \textbf{56}, 2597-2604.

\bibitem{ComPop} Comets, F.\ and Popov, S.\ (2007): On multidimensional branching random walks in random environment, {\em Ann. Probab.} {\bf 35}(1), 68-114.

\bibitem{pleasures} Deijfen, M.\ and H\"aggstr\"om, O.\ (2007): The pleasures and pains of studying the two-type Richardson model, in \emph{Analysis and Stochastics of Growth Processes and Interface Models}, Oxford University Press, 39-54.

\bibitem{comp_frogs} Deijfen, M., Hirscher, T.\ and Lopes, F. (2019): \emph{Electron. J. Probab.} 24, 1-17.


\bibitem{Eth} Etheridge, A.\ (2004): Survival and extinction in a locally regulated population, \emph{Ann. Appl. Probab.} \textbf{14}, 188-214.

\bibitem{HP1} H\"{a}ggstr\"{o}m, O. and Pemantle, R. (1998): First passage percolation and a model for competing spatial growth, \emph{J. Appl. Probab.} \textbf{35}, 683-692.



\bibitem{KS_03} Kesten, H. and Sidoravicius, V. (2003): Branching random walk with catalysts, \emph{Electron. J. Probab.} {\bf 8}, 1-51.

\bibitem{RW_competition} Kurkova, I., Popov, S. and Vachkovskaia, M. (2004): On infection spreading and competition between independent random walks, \emph{Electron. J. Probab.} \textbf{9}, 293-315.

\bibitem{Neu} Neuhauser, C. (1992): Ergodic theorems for the multitype contact process, \emph{Probab. Theory Related Fields} \textbf{91}, 467-506.

\bibitem{Shi} Shi, Z.\ (2015): \emph{Branching Random Walks},  \'{E}cole d'et\'{e} Saint-Flour XLII (2012), Lecture Notes in Mathematics 2151, Springer.

\bibitem{VA} Sidoravicius, V. and Stauffer, A. (2019): Multi-particle diffusion limited aggregation, \emph{Invent. Math.} \textbf{218}, 491-571.

\end{thebibliography}
\end{document}